\documentclass{article}
\usepackage{latexsym}
\usepackage[all]{xy}
\usepackage{amssymb,amsmath,amsthm,verbatim,color}
\usepackage{enumitem}
\def\ni{\noindent}
\def\Pic{\mathrm{Pic}}

\def\Ext{\mathrm{Ext}}
\def\Jac{\mathrm{Jac}}

\def\Orb{\mathrm{Orb}}

\def\bP{\mathbb{P}}

\def\lra{\longrightarrow}
\def\ra{\to}

\def\cO{\mathcal{O}}

\def\L{\mathcal{L}}

\def\I{\mathcal{I}}

\def\E{\mathcal{E}}

\def\cS{\mathcal{S}}
\def\dim{\mathrm{dim}}

\def\ol{\overline}

  \newtheorem{theorem}{Theorem}[section]
\newtheorem{lemma}[theorem]{Lemma}
\newtheorem{prop}[theorem]{Proposition}
\newtheorem{definition}[theorem]{Definition}
\newtheorem{corollary}[theorem]{Corollary}

\newtheorem{remark}[theorem]{Remark}

\newtheorem{subsec}[theorem]{}

\newenvironment{dedication}
        {\vspace{6ex}\begin{quotation}\begin{center}\begin{em}}
        {\par\end{em}\end{center}\end{quotation}}

\title{On Treibich-Verdier curves}
\author{Edoardo Sernesi }
\date{}

\begin{document}

\maketitle
\begin{dedication}
A Enrico Arbarello con amicizia e stima
\end{dedication}

\begin{abstract}
    We survey some properties of a class of curves lying on certain elliptic ruled surfaces, 
studied by  A. Treibich and J.L. Verdier in connection with elliptic solitons and KP equations. In particular we discuss their Brill-Noether generality, proved by A. Treibich, and we show that they are limits of hyperplane sections of K3 surfaces.
\end{abstract}

\section{Introduction}

This largely expository note is devoted to the description of some  properties of a class of algebraic curves studied by A. Treibich and J.L. Verdier in connection with elliptic solitons and KP equations. We call them \emph{Treibich-Verdier curves} (shortly \emph{TV-curves}) in what follows. The topic is wide and covers the singular case as well, but we only focus on some aspects related with Brill-Noether theory and K3 surfaces. In \cite{aT93b} Treibich showed    that TV-curves   are Brill-Noether general. Since then  for some time they have been the most concrete known class of  curves with such property, and only recently a new similar class has been discovered, namely the Du Val curves, which are even Petri-general (see \cite{ABFS16}).  From the moduli point of view   TV-curves are very special, being ramified covers of elliptic curves, constructed inside an   elliptic ruled surface;   they owe   their Brill-Noether generality to a specific relation, called ``tangentiality'', between  the elliptic curve and the Abel-Jacobi image of the curve itself    (see Prop. \ref{P:tangential}).  The proof given in \cite{aT93b} is quite ingenious,  combining tangentiality  with a formula of Fay's and elementary properties of the jacobian. We   reproduce it in outline in \S\ \ref{S:brillnoether}. In the final \S\ \ref{S:embedding}, the only novel part of this note, we prove that   TV-curves of any genus $g\ge 3$ are limits of  hyperplane sections of K3 surfaces. This is done by showing that the elliptic ruled surface containing them can be realized birationally as a surface in $\bP^g$ whose hyperplane sections are canonical models of   TV-curves, and that such surface can be smoothed    in $\bP^g$, thus producing K3 surfaces whose hyperplane sections specialize to any given    TV-curve (Theorem \ref{T:K3limit}).

\ni A few interesting questions remain unanswered. Firstly, one would like to know whether   TV-curves are Petri general: Treibich's proof of Brill-Noether generality  apparently  cannot be adapted to give informations on the Petri map. Another  question concerns the corank of the Wahl map. Corollary \ref{C:treibichwahl} shows that the Wahl map $\Phi_{\Gamma_g}$ of a   TV-curve $\Gamma_g$ of any genus $g\ge 3$ is not surjective, but it does not compute its corank. It would be interesting to carry out such computation, in view of the fact that there are no known examples of Brill-Noether general curves of genus $g \ge 13$ whose Wahl map has corank strictly larger than one (compare \cite{CDS20}, Question 2.14).  In fact   corank$(\Phi_{\Gamma_g})$ is    related with   the extendability of the surfaces $S_g$, a property  which seems difficult to detect. Moreover, if one knew that $\Phi_{\Gamma_g}$ has corank  equal to one then it would follow that TV-curves are not hyperplane sections of K3 surfaces but just limits of such  (compare \cite{AB17}, where this question is discussed).

\ni
We work over $\mathbb C$,
the field of complex numbers.

\ni
The structure of the paper is the following. In \S 2 we collect some standard computations that are needed in the rest of the paper. \S 3 is devoted to the property of  tangentiality and to its consequences for TV-curves. In \S 4 we outline Treibich's proof of the Brill-Noether generality, and in the final \S 5 we discuss the relation between TV-curves and K3 surfaces.
\medskip

\textbf{Acknowledgements.} I thank A. Bruno and R. Salvati Manni for  useful conversations, and A. Treibich for kindly providing copies of his publications and for enlightening remarks about this paper.


\section{Some computations}
 We fix 
a projective nonsingular curve $E$  of genus one, and a point $q\in E$ as the origin. We let
 \begin{equation}\label{E:E}
     0 \lra \cO_E \lra \E \lra \cO_E \lra 0
 \end{equation}
be the unique non-split extension in $\Ext^1(\cO_E,\cO_E)$, $X=\bP(\E)$  and $\pi: X\lra E$ the projection. Moreover we let $E_0\subset X$ be the minimal section, of self-intersection $E_0^2=0$,   $f=\pi^{-1}(q)$ and     $p=E_0 \cap f$. Note that $\cO_X(E_0)=\cO_X(1)$,     $K_X=\cO_X(-2E_0)$ and $\chi(\cO_X)=0$.

\begin{lemma}\label{L:Eq}
  Let $S^a\E$ be the $a$-th symmetric power of $\E$, $a \ge 1$. Then:
  \begin{itemize}
      \item[(i)] $S^a\E$ is indecomposable, $h^0(E,S^a\E)=1= h^1(E,S^a\E)$ and $S^a\E\cong S^a\E^\vee$.
  \item[(ii)] For every $k \ge 1$ we have
  $
  h^0(E,S^a\E(kq))=k(a+1), \quad h^1(E,S^a\E(kq))=0
  $
  \end{itemize}
\end{lemma}
\begin{proof}
(i)  By induction on $a$. For $a=1$ the claim is clear because  the exact sequence \eqref{E:E} being non-split, has non-zero coboundary as well as its dual. Therefore $\E\cong F_2$ (in the notations of \cite{mA57}), the unique, up to isomorphism,    indecomposable vector bundle of rank two and degree zero    such that $H^0(E,\E)\ne 0$. 
Assume $a\ge 2$. Then, by loc. cit., Theorem 9, $S^a\E\cong F_{a+1}$,  the unique, up to isomorphism, indecomposable vector bundle of rank $a+1$ with $H^0(E,S^a\E)\ne 0$.    From \eqref{E:E} we deduce a non-split exact sequence:
\begin{equation}\label{E:SE}
    0\lra S^{a-1}\E \lra S^a\E \lra \cO_E \lra 0
\end{equation}
corresponding to a generator of $H^1(E,S^{a-1}\E)$. 
It shows that  $h^i(E,S^a\E)=1$, $i=0,1$.  Moreover $S^a\E\cong S^a\E^\vee$ by the case $a=1$ just proved (or by loc. cit., Corollary 1).

(ii)
   We give the proof only in the case $k=1$: in the other cases it is similar. By induction on $a$. From the exact sequence:
    \begin{equation}\label{E:Eq}
     0 \lra \cO_E(q) \lra \E(q) \lra \cO_E(q)\lra 0   
    \end{equation}
   we deduce that $h^0(E,\E(q))=2$ and $h^1(E,\E(q))=0$, and this takes care of the case $a=1$. Now assume  $a \ge 2$.  From \eqref{E:Eq} we deduce the exact sequence
    \begin{equation}\label{E:Eq2}
        0 \lra S^{a-1}\E(q) \lra S^a\E(q) \lra \cO_E(q)\lra 0
    \end{equation}
   By the inductive hypothesis we have $h^0(E,S^{a-1}\E(q)=a$ and $h^1(E,S^{a-1}\E(q)=0$. Then the cohomology sequence of \eqref{E:Eq2} proves the induction step and the Lemma. 
\end{proof}

\begin{lemma}\label{L:Gamma}
 Let $g\ge 1$. Then:
 \begin{enumerate}[label=(\roman*)]
 \item
 $h^0(X,\cO_X(gE_0))=h^1(X,\cO_X(gE_0))=1$, \quad $h^2(X,\cO_X(gE_0))=0$.
      \item  $h^0(X,\cO_X(gE_0+f))=g+1, \quad h^1(X,\cO_X(gE_0+f))=h^2(X,\cO_X(gE_0+f))=0$.
     \item The base locus of the linear system $|gE_0+f|$ is $\{p\}$.
     \item The general $\Gamma\in |gE_0+f|$ is a nonsingular connected curve of genus $g$.
     \item Every nonsingular $\Gamma\in |gE_0+f|$ satisfies $\cO_\Gamma(\Gamma)=\omega_\Gamma(2p)$ and the image of the restriction map:
     $$
     H^0(X,\cO_X(\Gamma)) \lra H^0(\Gamma,\omega_\Gamma(2p))
     $$
     defines the linear system $|\omega_\Gamma|+2p$.

 \end{enumerate}   
\end{lemma}
\begin{proof}
(i) From the Leray spectral sequence we get:
    $$
    h^i(X,\cO_X(gE_0))=h^i(E,S^g\E)
    $$
   Then we conclude by Lemma \ref{L:Eq}(i).
   
    (ii) From the Leray spectral sequence we get:
    $$H^i(X,\cO_X(gE_0+f))=H^i(E,S^g\E(q))$$
    Then we conclude by Lemma \ref{L:Eq}(ii).

    (iii) Since $(gE_0+f)\cap E_0=p$, certainly the base locus $\mathrm{Bs}(|gE_0+f|)$ contains $p$. Consider the exact sequence 
     \begin{equation}\label{E:gammaff}
     0\lra \cO_X(gE_0)\lra \cO_X(gE_0+f) \lra \cO_f(gE_0+f)\lra 0    
     \end{equation}
     Lemma \ref{L:Eq}(i) and Leray spectral sequence imply that 
      $$
      h^1(X,\cO_X(gE_0))=h^1(E,S^g\E)=1
      $$
      and    Lemma \ref{L:Eq}(ii) gives $H^1(X,\cO_X(gE_0+f))=0$. Therefore we see that the map
     $$
     H^0(X,\cO_X(gE_0+f)) \lra H^0(f,\cO_f(gE_0+f))
     $$
     has corank one: thus $|gE_0+f|$ has   $p$ as a simple base point on $f$; in particular the general $\Gamma\in |gE_0+f|$ is not tangent to $f$. 
     Let now  $f'=\pi^{-1}(q')$, $q'\ne q$, be another fibre.   The exact sequence
    $$
    0\lra \cO_X((g-1)E_0+f-f')\lra \cO_X(gE_0+f-f') \lra \cO_{E_0}(gE_0+f-f')\lra 0
    $$
    and induction show that 
     \begin{equation}\label{E:fibrevanish}
         H^0(X,\cO_X(gE_0+f-f'))=H^1(X,\cO_X(gE_0+f-f')) =0
     \end{equation}
    for all $g \ge 1$. Therefore from   the following exact sequence:
     $$
     0\lra \cO_X(gE_0+f-f')\lra \cO_X(gE_0+f) \lra \cO_{f'}(gE_0+f) \lra 0
     $$
     we deduce that  the restriction map
     $$
     H^0(X,\cO_X(gE_0+f)) \lra H^0(f', \cO_{f'}(gE_0+f))
     $$
    is surjective. Therefore $|gE_0+f|$ has no base points on $f'$. This concludes the proof of (iii). 

    (iv)   Since $(gE_0+f) \cdot E_0=1$, we see that $\Gamma$ intersects $E_0$ at $p$ with multiplicity one, thus $\Gamma$ is nonsingular at $p$ and it is nonsingular elsewhere by Bertini. Moreover by adjunction we get:
    $$
    2g(\Gamma)-2 = ((g-2)E_0+f)\cdot (gE_0+f)=2g-2
    $$
    and $g(\Gamma)=g$.

    (v) We have 
    $$
    \cO_\Gamma(\Gamma) = \cO_\Gamma(\Gamma-2E_0)\otimes \cO_\Gamma(2E_0)= \omega_\Gamma\otimes\cO_\Gamma(2p)=\omega_\Gamma(2p)
    $$
    The last assertion follows easily by looking at the following diagram:
    $$
    \xymatrix{
0\ar[r]&\cO_X \ar[r]&\cO_X(\Gamma) \ar[r]&\omega_\Gamma(2p)\ar[r]&0 \\
0\ar[r]&\cO_X(-E_0)\ar[u]^-{E_0}\ar[r]&\cO_X((g-1)E_0+f)\ar[u]^-{E_0}\ar[r]&\omega_\Gamma(p)\ar[r]\ar[u]^-p&0
    }
    $$
\end{proof}

\begin{remark}\label{R:tangency}\rm
From  Lemma \ref{L:Gamma} it follows that all curves of $|\Gamma|$  have a fixed tangent at $p$ along a direction trasversal to both $E_0$ and $f$. 
\end{remark}
     
\begin{definition}\label{D:treibC}
    Let $g \ge 1$. A \emph{Treibich-Verdier curve} (shortly a \emph{TV-curve})  of genus $g$ is a smooth curve $\Gamma_g\in |gE_0+f|$.
\end{definition}

 Next we investigate some remarkable properties of TV-curves.


\section{Tangentiality}

\subsec\rm Let $\varphi: (C,p) \lra (D,q)$ be a   finite morphism of projective smooth connected pointed curves of positive genera, and let $\Jac(C)$ and $\Jac(D)$ be their jacobian varieties,  parametrizing isomorphism classes of invertible sheaves of degree zero. Then we have the following homomorphisms:
\begin{enumerate}
    \item $\varphi^*:\Jac(D) \lra \Jac(C)$, \quad $L \mapsto \varphi^*L$.
    \item The norm  map 
    $$
    \mathrm{Nm}(\varphi): \Jac(C)\lra \Jac(D), \quad M \mapsto  \det(\varphi_*M)\otimes \det(\varphi_*\cO_C)^{-1}
    $$
    \item The Abel maps 
    $$
    \mathrm{Ab}_p: C \lra \Jac(C), \quad x\mapsto \cO_C(x-p)
    $$
    and
    $$
    \mathrm{Ab}_q: D \lra \Jac(D), \quad y \mapsto \cO_D(y-q)
    $$
    \item\label{enum:etale} The composition
    $$
    \iota_\varphi := \varphi^*\cdot \mathrm{Ab}_q: D \lra \Jac(C)
    $$
    It is  etale onto its image, since $\mathrm{Nm}(\varphi)\cdot\varphi^*$ is multiplication by $\deg(\varphi)$.
\end{enumerate}

 \begin{definition}[\cite{aT89}, Def. 1.6]\label{Def:tang}
     $\varphi$ is called a \emph{tangential cover} if $\iota_\varphi(D)$ is tangent to $\mathrm{Ab}_p(C)$ at $0\in \Jac(C)$.
 \end{definition}
In \cite{aT89,TV89,TV91} various properties of tangential covers are studied, even in the singular case, mainly in the case when $D$ has genus one. We will  only recall what we need for the study of TV-curves. 

\subsec\rm We  go back to the notations introduced in the previous sections. 
 We fix $g\ge 3$ and a TV-curve of genus $g$, which  for short we denote by $\Gamma$. We denote by   
$$
\pi_\Gamma:=\pi_{|\Gamma}: (\Gamma,p)\lra (E,q)
$$
the degree $g$ cover of pointed curves obtained by restricting $\pi: X \lra E$ to $\Gamma$. 
   Then $\mathrm{Ab}_q(\alpha)=\cO_E(\alpha-q)$ and
$$
\iota_\pi: E \lra \Jac(\Gamma), \quad \alpha\mapsto \pi_\Gamma^*(\alpha)\otimes\pi_\Gamma^*(q)^{-1}
$$
By \ref{enum:etale} this map is etale onto its image;   it is an embedding if and only if $\pi_\Gamma$ does not factorize via a cyclic etale cover $h:E'\lra E$  (\cite{BL04}, Prop. 11.4.3).
 The  following is a fundamental property of $\pi_\Gamma$:  
 \begin{prop}[\cite{TV91}, Cor. 3.10]\label{P:tangential}
$\pi_\Gamma$ is a tangential cover.
 \end{prop}

 \begin{proof}
     The non-zero element $\eta\in H^1(E,\cO_E)$   corresponding to the extension \eqref{E:E} generates the tangent space $T_0\Jac(E)$. The pullback of $\eta$ by $\pi_\Gamma$
$$
\pi_\Gamma^*(\eta): 0\lra \cO_\Gamma \lra \pi_\Gamma^*\E\lra \cO_\Gamma \lra 0
$$
 identifies    $d\iota_{\pi}(\eta)\in T_0\Jac(\Gamma)=H^1(\Gamma,\cO_\Gamma)$ as the image of the cobounday map:
 $$
 \partial: H^0(\Gamma,\cO_\Gamma) \lra H^1(\Gamma,\cO_\Gamma)
 $$
 We must show that  $d\iota_{\pi}(\eta)$ is tangent to $\mathrm{Ab}_p(\Gamma)$, i.e. that $d\iota_{\pi}(\eta)$ generates $T_0\mathrm{Ab}_p(\Gamma)$. Firstly observe that $d\iota_{\pi}(\eta)\ne 0$, by \ref{enum:etale}.  It is classical and well known (compare\cite{ACGH85}) that $T_0\mathrm{Ab}_p(\Gamma)$ is   the kernel of  the map 
 $$
   H^1(\Gamma,\cO_\Gamma)\lra H^1(\Gamma,\cO_\Gamma(p))
 $$
 induced by multiplication with a non-zero section $\sigma_p\in H^0(\Gamma,\cO_\Gamma(p))$. Observe that on the surface $X$ we have a canonical quotient $\pi^*\E\lra \cO_X(C_0)=\cO_X(1)$ which restricts on $\Gamma$ to a quotient $\zeta:\pi_\Gamma^*\E \lra \cO_\Gamma(p)$. Consider the following diagram:
 $$
 \xymatrix{
 \pi_\Gamma^*(\eta):&0 \ar[r]&\cO_\Gamma\ar[d]^-{\sigma_p}\ar[r]&\pi_\Gamma^*\E \ar[r]\ar[d]\ar[dl]^-{\zeta}&\cO_\Gamma\ar[r]\ar@{=}[d]&0 \\
 &0 \ar[r]&\cO_\Gamma(p)\ar[r]&\E'\ar[r]&\cO_\Gamma\ar[r]&0
 }
 $$
 where the second row is the extension   obtained as the pushout of $\pi^*_\Gamma(\eta)$ under $\sigma_p$.  Taking cohomology and chasing the diagram we see that the existence of $\zeta$ implies  that 
 $$
 \xymatrix{
 d\iota_{\pi}(\eta)\in \ker[ H^1(\Gamma,\cO_\Gamma)\ar[r]^-{\sigma_p}& H^1(\Gamma,\cO_\Gamma(p))]
 }
 $$
 which is precisely what had to be proven.
 \end{proof}

 The map $\iota_\pi$   defines an action of $E$   on $\Jac(\Gamma)$.  Then, for any $d$, the map $\iota_\pi$ induces an action of $E$ on $\Pic^d(\Gamma)$, the variety of isomorphism classes of invertible sheaves of degree $d$, by restricting to $\iota_\pi(E)$ the action of $\Jac(\Gamma)$. We   focus on the case $d=g-1$. For each $\xi\in \Pic^{g-1}(\Gamma)$ we denote by
 $$
 \Orb_\xi: E \lra \Pic^{g-1}(\Gamma), \quad \alpha\mapsto \xi\otimes\iota_\pi(\alpha)=\xi\otimes \pi_\Gamma^*(\alpha)\otimes \pi_\Gamma^*(q)^{-1}
 $$
 Then $\Orb_\xi(E)$ is the \emph{orbit} of $\xi$ by the action of $E$ on $\Pic^{g-1}(\Gamma)$.  The following is an immediate consequence of Proposition \ref{P:tangential}:

 \begin{corollary}\label{C:tangential}
 $\Orb_\xi(E)$ is tangent to $\mathrm{Ab}_p(\Gamma)+\xi$ at $\xi$.    
 \end{corollary}
Let
 $$
 \Theta := \{\xi\in \Pic^{g-1}(\Gamma): h^0(\Gamma,\xi) \ge 1\}
 $$
 be the canonical \emph{theta divisor}. The following result is fundamental in the study of TV-curves.

 \begin{prop}[\cite{TV89}]\label{P:Ifinite}
  \begin{enumerate}[label=(\roman*)]
     \item For each $\xi\in \Pic^{g-1}(\Gamma)$ the orbit $\Orb_\xi(E)$ is not contained in $\Theta$, in particular  the pullback $\Orb^*_\xi(\Theta)$ is an effective divisor on $E$.

     \item Let $\xi \in \Theta$ and let $r+1=h^0(\Gamma,\xi)$. Then the intersection multiplicity of $\mathrm{Orb}(E)$ with $\Theta$ at $\xi$ is
     \begin{equation}\label{E:fay}
         \mathrm{mult}_\xi(\mathrm{Orb}(E),\Theta)= r+1+\sum_{i=1}^{r+1}(m_i+n_i)
     \end{equation}
     where the integers $m_1, \dots,m_{r+1},n_1,\dots,n_{r+1}$ are defined by the following identities:
     \begin{align*}
     &h^0(\Gamma,\xi(-m_ip))=&r+1-i+1&=h^0(\Gamma,\omega_\Gamma\xi^{-1}(-n_ip))    \\
     &h^0(\Gamma,\xi(-(m_i+1)p))=&r+1-i&=h^0(\Gamma,\omega_\Gamma\xi^{-1}(-(n_i+1)p))
     \end{align*}
     for all $i=1,\dots,r+1$.

     \item $\Orb^*_\xi(\Theta)$ has degree $g$ for all $\xi \in \Pic^{g-1}(\Gamma)$.
     \item the morphism 
 $$
 I: \Pic^{g-1}(\Gamma) \lra E^{(g)}, \quad \xi\mapsto \Orb^*_\xi(\Theta)
 $$
 is finite (where $E^{(g)}$ denotes the $g$-th \emph{symmetric product} of $E$).
 \end{enumerate}
 \end{prop}

\bigskip

\begin{proof}
    (i) Let $T=T_0(\mathrm{Ab}_p(\Gamma))$,   the tangent space at $0\in \Jac(\Gamma)$ of the Abel image of $\Gamma$. Let $G(\Gamma,p)\subset \Jac(\Gamma)$ be the 1-parameter subgroup generated by any  $0\ne v\in T$.  By \cite{jF84}, Th. 1.2, the orbit $\xi\cdot G(\Gamma,p)$ is not contained in $\Theta$, for any $\xi\in \Pic^{g-1}(\Gamma)$. Since $\pi_\Gamma$ is tangential, we have $T=T_0(\iota_\pi(E))$ and therefore $G(\Gamma,p)=\iota_\pi(E)$. Thus $\Orb_\xi(E)=\xi\cdot G(\Gamma,p)$ is not contained in $\Theta$ for all $\xi\in \Pic^{g-1}(\Gamma)$.

(ii) Since $G(\Gamma,p)=\iota_\pi(E)$ the multiplicity $\mathrm{mult}_\xi(\mathrm{Orb}(E),\Theta)$ coincides with the number $N$ appearing in the statement of  \cite{jF84}, Th. 1.2, i.e. with the right hand side of  \eqref{E:fay}.

(iii)   $\deg(\mathrm{Orb}^*_\xi(\Theta))$ does not depend on $\xi$; therefore it suffices to prove (iii) for just one choice of $\xi$: we take $\xi=(g-1)p$.
 Let's write $\epsilon=\cO_E(\alpha-q)\in \Jac(E)$. Since $\xi\otimes \iota_\pi(\alpha)=\xi\otimes \pi_\Gamma^*\epsilon$,   we have $\alpha\in \mathrm{Supp}\left(\mathrm{Orb}_\xi^*(\Theta)\right)$ if and only if 
\begin{equation}\label{E:piomega2}
H^0(E,\pi_{\Gamma*}\xi(\epsilon))=H^0(\Gamma,\xi\otimes\pi_\Gamma^*\epsilon) > 0
 \end{equation}
Consider the exact sequence on $X$:
$$
0 \lra \cO_X(-E_0-f)\lra \cO_X((g-1)E_0)\lra \cO_\Gamma((g-1)p)\lra 0
$$
taking direct images we get 
$$
\pi_{\Gamma*}\xi\cong\pi_*\cO_X((g-1)E_0)=S^{g-1}\E
$$
Since $$
H^0(E,S^{g-1}\E(\epsilon))=\begin{cases}1, \quad \epsilon=0\\0, \quad \epsilon \ne 0
    \end{cases}
$$ 
 we see that $h^0(\Gamma,\xi)=1$ and $\mathrm{Supp}\left(\mathrm{Orb}_\xi^*(\Theta)\right)=\{q\}$. We now apply  \eqref{E:fay}. The integers appearing in 
 \eqref{E:fay} are $r=0$ and $m_1=g-1$, $n_1=0$, because $p\notin \mathrm{Supp}(\omega_\Gamma \xi^{-1})$, since $\omega_\Gamma \xi^{-1}=f\cdot \Gamma-p$. Therefore we obtain
 $$
 \mathrm{mult}_\xi(\mathrm{Orb}_\xi(E),\Theta)=1+(g-1+0)=g
 $$
 and this proves that $\mathrm{Orb}^*_\xi(\Theta)=gq$.

\medskip
(iv) Consider $gq\in E^{(g)}$. If $\xi\in I^{-1}(gq)$ then 
 $\mathrm{Orb}_\xi(E)\cap \Theta = \{\xi\}$ and $\mathrm{mult}_\xi(\mathrm{Orb}_\xi(E),\Theta)=g$. Then $\xi=(g-1)p$. In particular $I^{-1}(gq)$ is finite, thus $I$ is generically quasi-finite, and therefore surjective.  Consider the Stein factorization of $I$:
 $$
 \xymatrix{
\Pic^{g-1}(\Gamma)\ar[dr]_-\phi \ar[rr]^-I&& E^{(g)} \\
&Z\ar[ur]_-\psi
 }
 $$
 where $\psi$ is finite and $\phi$ is birational with connected fibres. If $I$ is not finite then $\phi$ has a positive dimensional fibre $F\subset \Pic^{g-1}(\Gamma)$. This implies that $F$ contains a rational curve, and this is impossible, because $\Pic^{g-1}(\Gamma)$ is an abelian variety.
 \end{proof}


\section{Brill-Noether generality}\label{S:brillnoether}

 Let $C$ be a projective nonsingular irreducible curve of genus $g\ge 3$.   Given integers $r,d \ge 0$  there is a well defined closed subscheme    $W^r_d(C)\subset \Pic^d(C)$ defined set-theoretically by:
 $$
 W^r_d(C)=\{L\in \Pic^d(C): h^0(C,L) \ge r+1\}
 $$
 (see \cite{ACGH85} for details). Let
 $$
 \rho(g,r,d) = g - (r+1)(g-d+r)
 $$
 be the \emph{Brill-Noether number}. It is a classical result (\cite{K71,KL72}) that if $\rho(g,r,d) \ge 0$ then $W^r_d(C)\ne \emptyset$ and $\dim(W^r_d(C)) \ge \rho(g,r,d)$ on every curve of genus $g$. Recall the following 

 \begin{definition}
     The curve $C$ is \emph{Brill-Noether general} if for every $r,d$ as above we have:
     
     $\bullet$ $\dim(W^r_d(C)) = \rho(g,r,d)$ if $\rho(g,r,d)\ge 0$

     $\bullet$ $W^r_d(C)=\emptyset$  if $\rho(g,r,d)< 0$.
 \end{definition}
 From  Griffiths-Harris \cite{GH80} we know that Brill-Noether general curves of any genus exist.    Nevertheless it is notoriously difficult to produce concrete examples of such curves. 
 By a result of Lazarsfeld \cite{rL86} we know that Brill-Noether general curves can be found in certain linear systems on K3 surfaces. More explicit Brill-Noether general curves of any genus, the so-called \emph{Du Val curves},  have been constructed   in \cite{ABFS16}. It turns out that TV-curves  also enjoy such property.
 In fact we have the following:

 \begin{theorem}[\cite{aT93b}]\label{T:treibich}
     All TV-curves are Brill-Noether general.
 \end{theorem}

\begin{proof}
      We must show that 
$W^r_d(\Gamma)=\emptyset$  if $\rho:=\rho(g,r,d)<0$ and that $\dim(W^r_d(\Gamma))=\rho$ if $\rho\ge 0$.  Let $\L\in W^r_d(\Gamma)$ be such that $h^0(\Gamma,\L)=r+1$. Since 
$$
\rho(g,r,d)=\rho(g,g-d+r-1,2g-2-d)
$$ 
we may assume that $d \le g-1$. Let 
 $$
 \xi = \L((g-1-d)p)\in \Pic^{g-1}(\Gamma)
 $$
 We will apply formula \eqref{E:fay} to $\xi$. We have:
 $$
 h^0(\Gamma,\xi) = r+1+s
 $$
 for some $s \ge 0$. Then we have  the following obvious lower estimates:
      \begin{itemize}
          \item[(a)] $m_j\ge j-1$ if $j=1,\dots, s$,
          \item[(b)] $m_{s+k}\ge (g-1-d)+k-1$ if $k=1,\dots, r+1$.
          \item[(c)] $n_i\ge i-1$ for every $i=1, \dots, r+1+s$.
      \end{itemize}
    Substituting in \eqref{E:fay} we get:  
    
    \begin{align*}
     \mathrm{mult}_\xi(\Orb_
\xi(E), \Theta)\ge 
     &  (r+1+s)+\frac{1}{2}s(s-1)+(r+1)(g-1-d)+\\
     &+\frac{1}{2}r(r+1)+\frac{1}{2}(r+s)(r+s+1)\ge\\
     &\ge (r+1)(g-d+r)
    \end{align*}
    We cannot have $(r+1)(g-d+r)> g$ because otherwise $\Orb_\xi(E)\subset \Theta$, contradicting Proposition \ref{P:Ifinite}(i). Therefore  $\rho(g,r,d) \ge 0$. Moreover, since $$
    \mathrm{mult}_\xi(\Orb_
\xi(E), \Theta)\ge (r+1)(g-d+r) = g-\rho
$$ 
we have 
$$
I((g-1-d)p+W^r_d(\Gamma))\subset (g-\rho)q +E^{(\rho)}
$$
  From Proposition \ref{P:Ifinite}(iv)  we deduce that $\dim(W^r_d(\Gamma))\le \dim(E^{(\rho)})=\rho(g,r,d)$. But we also have $\dim(W^r_d(\Gamma))\ge \rho(g,r,d)$, and therefore we must have equality.
  \end{proof}

  \begin{remark}\label{R:FT1}\rm
    In \cite{FT17}, Theorem 4, the authors       prove  that a general TV-curve satisfies the 1-pointed Brill-Noether theorem, a slightly stronger notion than Brill-Noether generality. Their proof is by degeneration, thus completely different from Treibich's proof of Theorem \ref{T:treibich}.     
  \end{remark}

  \begin{remark}\label{R:TV1}\rm
       It is interesting to observe the analogy between the series of   TV-curves of genus $g$ on $X$ and the series of Du Val curves of genus $g$ on the blow-up $S$ of $\bP^2$ at nine general points \cite{ABFS16}: for a given $g$ the curves of both classes     are Brill-Noether general  and move in a $g$-dimensional linear system with one base point.  In fact Du Val curves are even Petri-general. It is not known whether TV-curves are Petri general, and the proof of Theorem \ref{T:treibich} does not seem to generalize to prove such stronger property.
  \end{remark}


  \section{Embedding $X$ birationally  in $\bP^g$ with canonical TV-curves of genus $g$ as hyperplane sections}\label{S:embedding}

We now fix $g \ge 3$ and we let $\Gamma_g\in |gE_0+f|$ be a general TV-curve of genus $g$. We  blow-up $X$ at the base point $p$ of $|gE_0+f|$.  We obtain
$$
\sigma: X' \lra X
$$
with $\sigma^{-1}(p)=: e'$. We let 
$$
E':=\sigma^*(E_0)-e', \quad   f':=\sigma^*(f)-e',   \quad \Gamma_g':=\sigma^*(\Gamma_g)-e'\sim gE'+f'+ge'
$$
   be the proper transforms of $E_0, f$ and  $\Gamma_g$ respectively.
The intersection numbers are:
$$\begin{aligned}
   E'^2=e'^2= f'^2 =-1,& \quad E'\cdot e'=e'\cdot f'=1\\
   E'\cdot f'=0, &\quad 
     \Gamma_g'\cdot E'=0,  & \Gamma_g'^2=2g-1
\end{aligned}
$$
Moreover $\Gamma_g'\cdot f'=0$ because $\Gamma_g$ is not tangent to $f$ (see Remark \ref{R:tangency}).

\noindent
We also have $K_{X'}=-2E'-e'$. Moreover $\cO_{\Gamma_g'}(\Gamma_g')=\omega_{\Gamma_g'}(p'(g))$, where $p'(g):=e'\cap \Gamma_g'$ is the point corresponding to the common tangent line at $p$ of all the curves of $|\Gamma_g|$  (Remark \ref{R:tangency}). Therefore the linear system $|\Gamma_g'|$ has a base point at $p'(g)$.

Let $\nu:Y_g\lra X'$ be the blow-up at $p'(g)$, $e=\nu^{-1}(p'(g))$ the exceptional curve, and
$$
\ol E =\nu^* E', \quad \ol e=\nu^*e'-e, \quad \ol\Gamma_g=\nu^*\Gamma_g'-e.
$$
Then $K_{Y_g}=-2\ol E-\ol e$, \ $\ol\Gamma_g\cdot K_{Y_g}=0$, \ $\cO_{\ol\Gamma_g}(\ol\Gamma_g)=\omega_{\ol\Gamma_g}$.

\begin{prop}\label{P:gorenstein}
\begin{itemize}
    \item[(i)] The linear system $|\ol\Gamma_g|$ is base point free and has dimension 
$$
\dim(|\ol\Gamma_g|)= h^0(Y,\cO_{Y_g}(\ol\Gamma_g))-1=g
$$
\item[(ii)]
    Let $S_g:=\varphi_{\ol\Gamma_g}(Y_g)\subset \bP^g$ be the image of $Y_g$ by the linear system $|\ol\Gamma_g|$. Then $S_g$ is a normal surface whose hyperplane sections are  canonically embedded TV-curves of genus $g$. $S_g$ has a unique Gorenstein singular point of geometric genus two and $\omega_{S_g}=\cO_{S_g}$.
    \end{itemize}
\end{prop}
\begin{proof}
(i) Since $\ol\Gamma_g=\nu^*(\Gamma_g')-e$,
we have, by Leray spectral sequence:
$$
h^0(Y_g,\cO_{Y_g}(\ol\Gamma_g))=h^0(X',\I_{p'(g)}(\Gamma_g'))=h^0(X',\cO_{X'}(\Gamma_g'))
$$
Similarly
$$
h^0(X',\cO_{X'}(\Gamma_g'))=h^0(X,\cO_X(\Gamma_g))=g+1
 $$
 By the Riemann-Roch theorem on $Y_g$ we get $\chi(\cO_{Y_g}(\ol\Gamma_g))=g-1$ and therefore $h^1(Y_g,\cO_{Y_g}(\ol\Gamma_g))=2$.
 From the exact sequence:
 $$
 0 \lra \cO_{Y_g} \lra \cO_{Y_g}(\ol\Gamma_g)\lra \omega_{\ol\Gamma_g}\lra 0
 $$
    we see that the characteristic linear system of $|\ol\Gamma_g|$ is the complete canonical system on $\ol\Gamma_g$. This     proves (i).
    
(ii)  An easy computation based on the  exact sequences:
$$
\xymatrix{&0\ar[d]\\
&\cO_{Y_g}(\ol E) \ar[d] \\
0\ar[r]&\cO_{Y_g}(\ol E+\ol e) \ar[d]\ar[r]&\cO_{Y_g}(-K_{Y_g}) \ar[r] & \cO_{\ol E}(-K_{Y_g})\lra 0 \\
&\cO_{\ol e}(\ol E+\ol e)\ar[d]\\
&0
}
$$
  shows that $h^0(Y_g,\cO_{Y_g}(-K_{Y_g}))=1$.
 Let $2\ol E+\ol e=\ol E+J$ be the unique element of $|-K_{Y_g}|$.
 Since $\ol\Gamma_g\cdot J=0$ and $J$ is connected, the linear system $|\ol\Gamma_g|$ contract $J$ to a normal singular point $z_g\in S_g$ and $S_g$ is a  surface with canonical curve sections. Therefore $S_g$ is arithmetically Gorenstein, hence normal. It is evident, by construction, that the members of $|\ol\Gamma_g|$  are TV-curves. 
 According to \cite{yU81}, Theorem 2,   $\omega_{S_g}=\cO_{S_g}$.   
 Consider the exact sequence coming from the Leray spectral sequence:
  $$
  H^1(S_g,\cO_{S_g})\ra H^1(Y_g,\cO_{Y_g})\ra H^0(S_g,R^1\varphi_*\cO_{Y_g})\ra H^2(S_g,\cO_{S_g})\ra H^2(Y_g,\cO_{Y_g})
  $$
Since $H^1(S_g,\cO_{S_g})=0$  and $H^2(Y_g,\cO_{Y_g})=H^0(Y_g,K_{Y_g})=0$ we deduce that:
$$
h^0(S_g,R^1\varphi_*\cO_{Y_g}) = h^1(Y_g,\cO_{Y_g})+h^2(S_g,\cO_{S_g}) =2
$$
 and therefore $z_g$ has geometric genus $p_g(z_g)=2$. According to \cite{yU81}, Theorem 1, $z_g$ is the unique singular point of $S_g$.
\end{proof}
 
   Of course  the surface $S_g\subset \bP^g$ described above  appears explicitely in the classification of normal surfaces with  canonical hyperplane sections, given in \cite{dE83} and in \cite{yU81}.    Next we are interested in deciding whether $S_g$ is a limit of K3 surfaces, i.e if it is smoothable in $\bP^g$, a question   considered in \cite{ABS17}. Our result is the following

   \begin{theorem}\label{T:K3limit}
       For a given $g\ge 3$ let $S_g\subset \bP^g$ be the surface constructed in Proposition \ref{P:gorenstein}.  Then $S_g$ is smoothable in $\bP^g$. Precisely there is a pointed affine nonsingular curve $(\Delta,0)$ and a flat family of projective surfaces
       $f:\cS_g \lra \Delta$  with  $\cS_g\subset \Delta\times \bP^g$ and $f$ induced by the projection, such that $\cS_g(0)=S_g$ and $\cS_g(t)\subset \bP^g$ is a nonsingular K3 surface of degree $2g-2$. Consequently every nonsingular hyperplane section of $S_g$ is  a  limit of hyperplane sections of K3 surfaces. 
   \end{theorem}

   \begin{proof}
On the surface  $X'$, the blow-up of $X$ at $p$, consider any point $p'\in e'\setminus \{E'\cap e', f'\cap e'\}$, and let $\nu:Y_{p'}\lra X'$ be the blow-up of $X'$ at $p'$. Moreover let $e\subset Y_{p'}$ be the exceptional curve, $\ol E=\nu^*E'$,\ \  $\ol e =\nu^*e'-e$. The intersection matrix
$$
\begin{pmatrix}\ol E^2&\ol E\cdot \ol e \\ \ol e\cdot \ol E&\ol e^2
\end{pmatrix} = \begin{pmatrix}
    -1&1\\1&-2
\end{pmatrix}
$$
   is negative definite. Therefore the curve $\ol E+ \ol e$ can be contracted in $Y_{p'}$ \cite{hG62}.    The resulting complex analytic surface $S_{p'}$ has an isolated singularity $z_{p'}$. Thus we obtain a flat family of surface singularities 
   $$
   \left\{z_{p'}: p'\in e'\setminus \{E'\cap e', f'\cap e'\}\right\}
   $$
    If $p'=p'(g)$  we get $z_{p'}=z_g$. In particular, if $p'=p'(3)$ then $S_{p'}=S_3$ is a quartic surface in $\bP^3$. Then $z_{p'}=z_3$ is a hypersurface singularity, hence it is smoothable and unobstructed.\footnote{$z_3$ is a double point, and appears as case (i) in Proposition 2 of \cite{yU84}. } 
    It follows that   all the $z_{p'}$'s are smoothable singularities. In particular $z_g$ is smoothable for all $g \ge 3$.
            By \cite{ABS17}, Th. 10.3, we deduce that all the surfaces $S_g$ are globally smoothable in $\bP^g$. 
            Note that $H^2(S_g,T_{S_g})=0$ can be proved exactly as in \cite{ABS17}, Lemma 10.1.
            
            \noindent
             Finally observe that, since the linear systems $|\cO_{S_g}(1)|$ and $|\cO_{\cS_g(t)}(1)|$ have the same dimension $g$, all curves of $|\cO_{S_g}(1)|$ extend to the general fibre of the smoothing family.
        \end{proof}

 We conclude with a corollary concerning the Wahl map of TV-curves. Recall that the  \emph{Wahl map} of a projective nonsingular curve $C$ of genus $g \ge 3$ is a natural linear map
 $$
 \Phi_C: \bigwedge^2H^0(C,\omega_C) \lra H^0(C,\omega_C^3)
 $$
 which extends linearly the application   sending a bivector $u\wedge v$ to (a section of $H^0(C,\omega_C^3)$ vanishing on) the ramification divisor of the pencil $\langle u,v\rangle \subset |\omega_C|$. For details on $\Phi_C$ we refer the reader to \cite{jW87}.

\begin{corollary}\label{C:treibichwahl}
All   TV-curves of genus $g \ge 3$ are limits of hyperplane sections of K3 surfaces and have non-surjective Wahl map.
\end{corollary}

\begin{proof}
    The first part is a restatement  of Theorem 
 \ref{T:K3limit}.      It is well known that hyperplane sections of K3 surfaces have non-surjective Wahl map \cite{jW87,BM87}. Then      TV-curves, being limits of curves with non-surjective  Wahl map, have non-surjective Wahl map as well, by semicontinuity.
\end{proof}


 \medskip\noindent
\textsc{Dipartimento di Matematica e Fisica \\ Universit\`a Roma Tre \\ L.go S.L. Murialdo, 1 -  00146 Roma, Italia.}\\
\texttt{sernesi@gmail.com}

\end{document}